\documentclass[10pt]{article}

\usepackage{amsthm,amsmath,amssymb}

\usepackage{graphicx}

\usepackage[colorlinks=true,citecolor=black,linkcolor=black,urlcolor=blue]{hyperref}

\theoremstyle{plain}
\newtheorem{theorem}{Theorem}[section]
\newtheorem{lemma}[theorem]{Lemma}

\newtheorem{proposition}[theorem]{Proposition}

\theoremstyle{definition}
\newtheorem{definition}[theorem]{Definition}
\newtheorem{example}[theorem]{Example}

\theoremstyle{remark}
\newtheorem{remark}[theorem]{Remark}

\title{\bf Generalized Dehn-Sommerville relations for hypergraphs}

\author{Vladimir Gruji\'c\\
\small Faculty of Mathematics\\[-0.8ex]
\small Belgrade University\\[-0.8ex]
\small Serbia\\
\small\tt vgrujic@matf.bg.ac.rs\\
\and
Tanja Stojadinovi\'c\\
\small Faculty of Mathematics\\[-0.8ex]
\small Belgrade University\\[-0.8ex]
\small Serbia\\
\small\tt tanjas@matf.bg.ac.rs
\and
Du\v{s}ko Joji\'c\\
\small Faculty of Science\\[-0.8ex]
\small University of Banja Luka\\[-0.8ex]
\small Bosnia and Herzegovina\\
\small\tt ducci68@blic.net}

\date{
}

\begin{document}

\maketitle

\begin{abstract}

The generalized Dehn-Sommerville relations determine the odd
subalgebra of the combinatorial Hopf algebra. We introduce a class
of eulerian hypergraphs that satisfy the generalized
Dehn-Sommerville relations for the combinatorial Hopf algebra of
hypergraphs. We characterize a wide class of eulerian hypergraphs
according to the combinatorics of underlying clutters. The
analogous results hold for simplicial complexes by the isomorphism
which is induced from the correspondence of clutters and
simplicial complexes.


\end{abstract}

\section{Introduction }

The combinatorial Hopf algebra is a graded connected Hopf algebra
with the multiplicative linear functional. The theory of
combinatorial Hopf algebras is established in \cite{ABS}. The
motivation and terminology come from the main example of the
Rota's Hopf algebra of finite graded posets \cite{JR}. The
generalized Dehn-Sommerville relations determine the odd
subalgebra of the combinatorial Hopf algebra. In the case of
Rota's Hopf algebra of posets the generalized Dehn-Sommerville
relations are precisely the relations introduced by Bayer and
Billera for the flag $f-$vector of a finite graded poset
\cite{BB}. The eulerian posets satisfy these relations and
generate the proper subalgebra of the odd subalgebra of posets.

The combinatorial Hopf algebra of hypergraphs extends the
chromatic Hopf algebra of graphs. The chromatic symmetric function
of a hypergraph, firstly introduced in \cite{SS}, appears as the
image of the canonical morphism to quasi-symmetric functions. The
classes of hypergraphs closed under disjoint sums and restrictions
form Hopf subalgebras of hypergraphs. Clutters and building sets
are objects of this type. They have an important role in
combinatorial algebra and geometry.

The combinatorial Hopf algebra of building sets is extensively
studied in \cite{GS}. All the main properties of a building set,
hypergraph and simplicial complex essentially depend on the
underlying clutters: minimal elements of the generating
collection, minimal edges and minimal non-faces. We use this fact
to elaborate the results from \cite{GS} for building sets to
hypergraphs and simplicial complexes. Also we correct the
deficiency in the formulation of \cite[Theorem 8.11]{GS}, where
the characterization of eulerian building sets is given in terms
of combinatorics of underlying clutters.

We say that a hypergraph is eulerian if the Euler character is
annihilated at all induced subhypergraphs. The eulerian
hypergraphs generate the proper subalgebra of the odd subalgebra
of hypergraphs. A hypergraph is eulerian if and only if the
clutter of its minimal edges is eulerian. To a clutter is
associated the intersection complex called the nerve. We obtain
that an odd clutter whose the nerve is a clique complex of a
chordal graph is eulerian. The converse is true under additional
condition that any proper subclutter is also a proper restriction.
Another characterization follows from Stanley's expansion of the
chromatic symmetric function of a hypergraph into the power sum
symmetric functions.

There is a natural correspondence of clutters and simplicial
complexes. To a clutter is associated the independence complex, by
analogy to independence complexes of simple graphs. On the other
hand, a simplicial complex determines the clutter of its minimal
nonfaces. A coloring of a clutter is equivalent to a partitioning
of its independence complex. Henceforth the canonical morphism to
a simplicial complex associates the generating function of its
partition functions. The eulerian simplicial complexes correspond
to independence complexes of eulerian clutters.

\section{Preliminaries from combinatorial Hopf algebras}

We fix some notations and review basic definitions and facts from
the theory of combinatorial Hopf algebras developed in \cite{ABS}.
A {\it composition} $\alpha\models n$ of an integer
$n\in\mathbb{N}$ is a sequence $\alpha=(a_1,\ldots,a_k)$ of
positive integers with $a_1+\cdots+a_k=n$. Let $k(\alpha)=k$ be
the length of $\alpha$. A {\it partition} $\lambda\vdash n$ is a
composition $\lambda\models n$ with $a_1\geq a_2\geq\cdots\geq
a_k$.

The {\it combinatorial Hopf algebra} $(\mathcal{H},\zeta)$ (CHA
for short) over a field $\mathbf{k}$ is a graded connected Hopf
algebra $\mathcal{H}=\oplus_{n\geq 0}\mathcal{H}_n$ equipped with
the multiplicative linear functional
$\zeta:\mathcal{H}\rightarrow\mathbf{k}$ called the {\it
character}. The terminal object in the category of CHA's is the
CHA of quasisymmetric functions $(QSym,\zeta_Q)$. The basic
reference for quasisymmetric functions is \cite{EC}. Let
$\{M_\alpha\}_{\alpha\models n, n\in\mathbb{N}}$ be the linear
basis of monomial quasisymmetric functions, where

\[M_{(a_1,\ldots,a_k)}=\sum_{i_1<\cdots<
i_k}x_{i_1}^{a_1}\cdots x_{i_k}^{a_k}.\] The character $\zeta_Q$
on the monomial basis is defined by
$\zeta_Q(M_\alpha)=\left\{\begin{array}{cc} 1, & \alpha=(n) \
\mbox{or} \ ()\\ 0, & \mbox{otherwise}\end{array}\right..$ The
unique canonical morphism
$\Psi:(\mathcal{H},\zeta)\rightarrow(QSym,\zeta_Q)$ is given on
homogeneous elements by

\begin{equation}\label{canonical}
\Psi(h)=\sum_{\alpha\models n}\zeta_\alpha(h)M_\alpha, \
h\in\mathcal{H}_n,
\end{equation} where $\zeta_\alpha$ is the convolution product

\begin{equation}\label{convprod}
\zeta_\alpha=\zeta_{a_1}\cdots\zeta_{a_k}:\mathcal{H}
\stackrel{\Delta^{(k-1)}}\longrightarrow\mathcal{H}^{\otimes
k}\stackrel{proj}\longrightarrow\mathcal{H}_{a_1}\otimes\cdots\otimes\mathcal{H}_{a_k}
\stackrel{\zeta^{\otimes k}}\longrightarrow\mathbf{k}.
\end{equation} The algebra of symmetric functions $Sym$ is the subalgebra of
$QSym$ and it is the terminal object in the category of
cocommutative combinatorial Hopf algebras.

A graded connected bialgebra $\mathcal{H}$ possesses the antipode

\begin{equation}\label{antipode}
S=\sum_{k\geq 0}(-1)^km^{(k-1)}\circ\pi^{\otimes k}
\circ\Delta^{(k-1)},
\end{equation} where $\pi=I-u\epsilon$ and $I,u,\epsilon$ are identity map,
unit and counit of $\mathcal{H}$ respectively. The convolution
inverse $\zeta^{-1}$, called the {\it M$\ddot{o}$bius character}
of $\mathcal{H}$, is the composition $\zeta^{-1}=\zeta\circ S$.
Let $\overline{\zeta}$ be the {\it conjugate character} defined on
homogeneous elements by $\overline{\zeta}(h)=(-1)^n\zeta(h), \
h\in\mathcal{H}_n$. Then the {\it Euler character} $\chi$ of CHA
$(\mathcal{H},\zeta)$ is defined by
\[\chi=\overline{\zeta}\zeta:\mathcal{H}\stackrel{\Delta}\longrightarrow\mathcal{H}\otimes\mathcal{H}
\stackrel{\overline{\zeta}\otimes\zeta}\longrightarrow\mathbf{k}\otimes\mathbf{k}\stackrel{m}\rightarrow\mathbf{k}.\]
The {\it odd subalgebra} $S_{-}(\mathcal{H},\zeta)$ is the largest
graded Hopf subalgebra of $\mathcal{H}$ on which the character is
odd $\zeta^{-1}=\overline{\zeta}$. A morphism
$\phi:(\mathcal{H}_1,\zeta_1)\rightarrow(\mathcal{H}_2,\zeta_2)$
preserves the odd subalgebra
$\phi(S_-(\mathcal{H}_1,\zeta_1))\subset
S_-(\mathcal{H}_2,\zeta_2).$ The {\it generalized Dehn-Sommerville
relations} for CHA $(\mathcal{H},\zeta)$

\begin{equation}\label{gdsr}
(id\otimes(\overline{\zeta}-\zeta^{-1})\otimes
id)\circ\Delta^{(2)}(h)=0
\end{equation} describe the sufficient and necessary conditions for
homogeneous elements to belong to the odd subalgebra
$S_-(\mathcal{H},\zeta)$. Given a homogeneous element $h\in
S_-(\mathcal{H},\zeta)$ of the rank $n$, it follows from
$(\ref{canonical})$ and the generalized Dehn-Sommerville relations
for quasi-symmetric functions (\cite{ABS}, Example 5.10) that

\begin{equation}\label{relation}
\sum_{j=0}^{a_i}(-1)^{j}\zeta_{(a_1,\ldots,a_{i-1},j,a_i-j,a_{i+1},\ldots,a_k)}(h)=0,
\end{equation}
for each $\alpha=(a_1,\ldots,a_k)\models n$ and
$i\in\{1,\ldots,k\}$.

\paragraph{Rota's Hopf algebra of posets}
The set of isomorphism classes of finite graded posets
$\mathbf{P}$ spans the Hopf algebra of posets
$\mathcal{H}(\mathbf{P})$ with the product $P_1\cdot P_2=P_1\times
P_2$ and the coproduct

\[\Delta(P)=\sum_{0_P\leq x\leq 1_P}[0,x]\otimes[x,1].\] Let $\zeta(P)=1$ for any poset $P$.
The antipode formula $(\ref{antipode})$ gives

\[S(P)=\sum_{k\geq 1}(-1)^k\sum_{0<x_1<\cdots<x_{k-1}<1}[0,x_1]\otimes\cdots\otimes[x_{k-1},1].\]
The M$\ddot{\mathrm{o}}$bius character $\zeta^{-1}=\zeta\circ S$
is the classical M$\ddot{\mathrm{o}}$bius function on posets. A
poset $P$ is eulerian if and only if the Euler character is
annihilated at all closed intervals $\chi([x,y])=0, x<y\in P$.

The canonical morphism
$\Psi:(\mathcal{H}(\mathbf{P}),\zeta)\rightarrow(QSym,\zeta_Q)$ to
a poset $P$ assigns the generating function of its flag $f-$vector

\[\Psi(P)=\sum_{\alpha\models n}f_\alpha(P)M_\alpha.\] The
relations $(\ref{relation})$ are precisely the Bayer-Billera
relations for the flag $f-$vector of an eulerian poset $P$.

\section{Hopf algebra of hypergraphs}

A {\it hypergraph} $H$ on the finite vertex set $V$ is a
collection $H\subset\mathcal{P}(V)$ of nonempty subsets of $V$.
The subsets $e\in H$ are called edges and we require that any edge
has at least two vertices. The hypergraph $D$ with no edges is the
discrete hypergraph on $V$. Two hypergraphs $H_1$ and $H_2$ are
isomorphic if there is a bijection $f:V(H_1)\rightarrow V(H_2)$
such that $e\in H_1$ if and only if $f(e)\in H_2$. A restriction
of the hypergraph $H$ on the vertex set $I\subset V(H)$ is the
hypergraph $H|_I=\{e\in H | e\subset I\}$.

Let $\mathcal{H}(\mathbf{H})=\oplus_{n\geq
0}\mathcal{H}_n(\mathbf{H})$ be the graded $\mathbf{k}-$vector
space linearly spanned by the set $\mathbf{H}$ of all isomorphism
classes of finite hypergraphs, where the grading is given by the
number of vertices. Define the product and the coproduct by

\[H_1\sqcup H_2=\{e\subset V(H_1)\sqcup V(H_2) \ | \  e\in H_1 \
\mbox{or} \ e\in H_2\}, \] \[ \Delta(H)=\sum_{I\subset
V(H)}H|_I\otimes H|_{V(H)\setminus I}.\] Then the space
$\mathcal{H}(\mathbf{H})$ is a graded connected commutative and
cocommutative Hopf algebra. The unit element is the hypergraph
$H_\emptyset$ on the empty set of vertices and the
counit is defined by $\epsilon(H)=\left\{\begin{array}{ll} 1, & H=H_\emptyset \\
0, & \mbox{otherwise}\end{array}\right..$ The antipode of
$\mathcal{H}(\mathbf{H})$ is determined by $(\ref{antipode})$ with

\begin{equation}\label{antiph}
S(H)=\sum_{k\geq 1}(-1)^k\sum_{I_1\sqcup\ldots\sqcup
I_k=V}H|_{I_1}\sqcup\ldots\sqcup H|_{I_k},
\end{equation}
where the sum goes over all ordered decompositions of the vertex
set.

Let $\zeta_\mathbf{H}(H)=\left\{\begin{array}{ll} 1, & H \ \mbox{is discrete} \\
0, & \mbox{otherwise}\end{array}\right.$. We obtain the CHA of
hypergraphs $(\mathcal{H}(\mathbf{H}), \zeta_\mathbf{H})$. The
canonical morphism $(\ref{canonical})$ of
$\mathcal{H}(\mathbf{H})$ to quasisymmetric functions is
determined by

\[\Psi(H)=\sum_{\alpha\models
n}(\zeta_\mathbf{H})_\alpha(H)M_\alpha, \
H\in\mathcal{H}_n(\mathbf{H}).\] For a composition
$\alpha=\{a_1,\ldots,a_k\}\models n$ the coefficient
$(\zeta_\mathbf{H})_\alpha(H)$ is the number of ordered
decomposition $(I_1,\ldots,I_k)$ of the vertex set $V$ such that
$H|_{I_j}$ is discrete of the rank $a_j$ for all $j=1,\ldots,k$.
Since $\mathcal{H}(\mathbf{H})$ is a cocommutative Hopf algebra,
the function $\Psi(H)$ is symmetric. It is easily seen to be the
generating function $\Psi(H)=\sum_{f}\prod_{v\in V}x_{f(v)}$ of
{\it proper colorings} $f:V\rightarrow \mathbb{N}$ of a hypergraph
$H$. Recall that a coloring $f:V\rightarrow\mathbb{N}$ is proper
if it is not monochromatic on edges of $H$.

The principal specialization $\chi(H,m)=\Psi(H)(1^m)$ defines the
{\it chromatic polynomial} of a hypergraph $H$. The value
$\chi(H,m)$ is the number of proper colorings of the hypergraph
$H$ by at most $m$ colors. By definition of the character
$\zeta_\mathbf{H}$ we have $\chi(H,m)=\zeta_\mathbf{H}^m(H)$. It
follows from the antipode formula $(\ref{antiph})$ that

\begin{equation}\label{-1}
\chi(H,-1)=\zeta_\mathbf{H}^{-1}(H)=\zeta_\mathbf{H}\circ
S(H)=\sum_{\alpha\models
n}(-1)^{k(\alpha)}(\zeta_\mathbf{H})_\alpha(H).
\end{equation}
Recall that if $H$ is a simple graph, the chromatic polynomial
$\chi(H,m)$ calculated at $m=-1$ gives the number of acyclic
orientations of $H$.

Let $\mathbf{S}$ be a family of isomorphism classes of hypergraphs
closed under taking disjoint sums and restrictions. The
$\mathbf{k}-$vector space $\mathcal{H}(\mathbf{S})$ with the basis
$\mathbf{S}$ and with the character
$\zeta_\mathbf{S}=\zeta|_{\mathcal{H}(\mathbf{S})}$ is a
combinatorial Hopf subalgebra of the algebra of hypergraphs
$(\mathcal{H}(\mathbf{H}),\zeta_\mathbf{H})$.

\begin{example}
Let $\mathbf{D}=\{D_n\}_{n\geq 0}$ be a family of discrete
hypergraphs. Then $\mathcal{H}(\mathbf{D})$ is isomorphic to the
polynomial Hopf algebra $\mathbf{k}[x]$, where $x=D_1$. The
character $\zeta_1=\zeta|_{\mathcal{H}(\mathbf{D})}$ is defined by
$\zeta_1(p)=p(1), p\in\mathbf{k}[x]$.
\end{example}

\begin{example}
Let $\mathbf{G}$ be the class of simple graphs considered as
$2-$uniform hypergraphs. Then
$(\mathcal{H}(\mathbf{G}),\zeta_\mathbf{G})$ is the chromatic Hopf
algebra of graphs.
\end{example}

We distinguish two important classes of hypergraphs that are
closed under restrictions and disjoint unions.

\paragraph{Clutters}

A {\it clutter} $C$ on the finite set $V$ is a hypergraph which is
an antichain in the boolean poset $\mathcal{P}(V)$ of all subsets
of $V$. For a hypergraph $H$ denote by $C(H)$ the clutter of
minimal edges of $H$. The clutters are naturally identified with
simplicial complexes. A simplicial complex $K$ defines the clutter
of its maximal faces. The family of all isomorphism classes of
clutters $\mathbf{C}$ defines the combinatorial Hopf algebra of
clutters $(\mathcal{H}(\mathbf{C}),\zeta_\mathbf{C})$. The obvious
inclusions are monomorphisms of CHA
\[(\mathcal{H}(\mathbf{D}),\zeta_\mathbf{D})\subset(\mathcal{H}(\mathbf{G}),\zeta_\mathbf{G})\subset(\mathcal{H}(\mathbf{C}),\zeta_\mathbf{C}).\]

\paragraph{Building sets}

A collection $B$ of non-empty subsets of the vertex set $V$ is a
{\it building set} on $V$ if it satisfies

$\diamond$ If $e, e'\in B$ and $e\cap e'\neq\emptyset$ then $e\cup
e'\in B$

$\diamond$ $\{v\}\in B$ for all $v\in V$.

\noindent We identify a building set $B$ with the hypergraph
$B\setminus\{\{v\} \ | \ v\in V\}$. The family of all isomorphism
classes of building sets $\mathbf{B}$ defines the combinatorial
Hopf algebra of building sets
$(\mathcal{H}(\mathbf{B}),\zeta_\mathbf{B})$. This algebra is
studied in \cite{GS}.

Let $p:\mathcal{H}(\mathbf{H})\rightarrow\mathcal{H}(\mathbf{C})$
be the projection map which assigns the clutter $C(H)$ of minimal
edges to a hypergraph $H$. Since $C(H_1\sqcup H_2)=C(H_1)\sqcup
C(H_2)$ for $H_1, H_2\in\mathbf{H}$ and $C(H|_I)=C(H)|_I$ for
$H\in\mathbf{H}$ and any $I\subset V(H)$, the map $p$ is a Hopf
algebra epimorphism. It holds that $p\circ
i=1_{\mathcal{H}(\mathbf{C})}$, where
$i:\mathcal{H}(\mathbf{C})\rightarrow\mathcal{H}(\mathbf{H})$ is
the inclusion map. Thus $i$ is a splitting monomorphism.

The odd subalgebra of hypergraphs $S_-(\mathcal{H}(\mathbf{H}),
\zeta_\mathbf{H})$ is determined by the generalized
Dehn-Sommerville relations $(\ref{gdsr})$. Since the inclusion
$i:(\mathcal{H}(\mathbf{C}),\zeta_\mathbf{C})\rightarrow(\mathcal{H}(\mathbf{H}),\zeta_\mathbf{H})$
is a morphism of CHA's, we have
$S_-(\mathcal{H}(\mathbf{C}),\zeta_\mathbf{C})=S_-(\mathcal{H}(\mathbf{H}),\zeta_\mathbf{H})\cap\mathcal{H}(\mathbf{C})$.

\begin{definition}
A hypergraph $H$ is eulerian if
$\chi_\mathbf{H}(H|_I)=\epsilon(H|_I)$ for all $I\subset V(H)$,
where
$\chi_\mathbf{H}=\overline{\zeta}_\mathbf{H}\zeta_\mathbf{H}$ is
the Euler character of CHA
$(\mathcal{H}(\mathbf{H}),\zeta_\mathbf{H})$.
\end{definition}

The following characterization is just a reformulation of the
definition.

\begin{lemma}\label{eulerII}
A hypergraph $H\neq H_\emptyset$ is eulerian if and only if either
$H|_I$ is descrete or $\zeta_\mathbf{H}^{-1}(H|_I)=0,$ for all
$\emptyset\neq I\subset V(H)$.
\end{lemma}

\begin{proof}
The condition $\chi_\mathbf{H}(H)=\epsilon(H)$ is equivalent to
$\zeta_\mathbf{H}^{-1}(H)=\overline{\zeta}_\mathbf{H}(H)$. Let
$\chi(D_n,m)=m^n$ be the chromatic polynomial of the discrete
hypergraph $D_n$ on $n$ vertices. By $(\ref{-1})$ we have
$\zeta^{-1}(D_n)=\chi(D_n,-1)=(-1)^n=\overline{\zeta}(D_n)$. As
$\overline{\zeta}_\mathbf{H}(H)=0$ for the nondiscrete $H$, the
lemma follows.

\end{proof}
Disjoint sums and restrictions of eulerian hypergraphs are
eulerian as well, so the family of all eulerian hypergraphs form
the Hopf subalgebra $\mathcal{E}(\mathbf{H})$ of CHA
$\mathcal{H}(\mathbf{H})$,  which we call the {\it eulerian
subalgebra} of hypergraphs.

\begin{proposition}\label{prop}
The eulerian subalgebra $\mathcal{E}(\mathbf{H})$ is a subalgebra
of the odd algebra
$S_-(\mathcal{H}(\mathbf{H}),\zeta_\mathbf{H})$. The coefficients
$(\zeta_\mathbf{H})_\alpha(H), \alpha\models|V(H)|$ of an eulerian
hypergraph $H\in\mathcal{E}(\mathbf{H})$ satisfy relations
$(\ref{relation})$.
\end{proposition}

\begin{proof}

The odd subalgebra $S_-(\mathcal{H}(\mathbf{H}),\zeta_\mathbf{H})$
is characterized by the generalized Dehn-Sommerville relations
$(\ref{gdsr})$
\[(id\otimes(\zeta^{-1}-\overline{\zeta})\otimes
id)\circ\Delta^{(2)}(H)=\sum_{I\sqcup J\sqcup
K=V(H)}H|_I\otimes(\zeta^{-1}-\overline{\zeta})(H|_J)\otimes
H|_K=\] $\sum_{J\subset
V(H)}(\zeta^{-1}-\overline{\zeta})(H|_J)\Delta(H|_{J^{c}})=0.$ It
follows that $H\in S_-(\mathcal{H}(\mathbf{H}),\zeta_\mathbf{H})$
if $(\zeta^{-1}-\overline{\zeta})(H|_J)=0$ for all $J\subset X$,
which is equivalent to $H\in\mathcal{E}(\mathbf{H})$ by lemma
$\ref{eulerII}$. The relations $(\ref{gdsr})$ imply
$(\ref{relation})$.

\end{proof}

\noindent Let
$\mathcal{E}(\mathbf{C})=\mathcal{E}(\mathbf{H})\cap\mathcal{H}(\mathbf{C})$
be the eulerian subalgebra of clutters.

\begin{proposition}\label{hc}
A hypergraph $H$ is eulerian if and only if the clutter $C(H)$ is
eulerian.
\end{proposition}

\begin{proof}

The proposition follows from the identity
$\chi_\mathbf{H}(H)=\chi_\mathbf{C}(C(H)), H\in\mathbf{H},$ which
is a consequence of the fact that the projection $p$ is a morphism
of CHA's, i.e. $\zeta_\mathbf{H}=\zeta_\mathbf{C}\circ p$.

\end{proof}

The main computational tool is the {\it deletion-contraction
property}. The chromatic symmetric function of a hypergraph $H$
depends only on the clutter of minimal edges, i.e.
$\Psi(H)=\Psi(C(H))$. Therefore $\chi(H,m)=\chi(C(H),m)$ for
chromatic polynomials. For a clutter $C$ on the vertices $V$ and
an edge $e\in C$ define $C-e:=C\setminus\{e\}$ to be {\it
deletion} of $e$ from $C$ and $C/e=\{e'/e | e'\in C-e\}$ to be
{\it contraction} of $C$ by $e$, where $C/e$ is the clutter on
$V\setminus e\cup\{e\}$ and $e'/e=\left\{\begin{array}{cc} e', & \
e'\cap
e=\emptyset \\
e'\setminus e\cup\{e\}, & e'\cap e\neq\emptyset
\end{array}\right..$ Then
\[\chi(C,m)=\chi(C-e,m)-\chi(C/e,m),\] and specially for $m=-1$ we obtain by $(\ref{-1})$
\begin{equation}\label{dc}
\zeta^{-1}(C)=\zeta^{-1}(C-e)-\zeta^{-1}(C/e). \end{equation}

To a clutter $C$ is associated the abstract simplicial complex on
the vertex set $C$, called the {\it nerve} $N(C)=\{C'\subset C |
\cap C'\neq\emptyset\}.$ The {\it intersection graph} $G(C)$ of a
clutter is the $1$-skeleton $N(C)^{(1)}$ of the nerve. We say that
$C$ is an {\it odd clutter} if $\cap C'$ is an odd element set for
all faces $C'\in N(C)$ of the nerve. The wide class of eulerian
clutters is identified by the following theorem which is proved
for building sets in \cite{GS}.

\begin{theorem}{\cite[Proposition 8.10]{GS}}\label{clique}
Let $C$ be an odd clutter such that
\begin{itemize}
\item[(i)] The graph $G(C)$ is a chordal graph \item[(ii)] The
nerve $N(C)$ is the clique complex of $G(C)$.
\end{itemize}
Then $C$ is an eulerian clutter.
\end{theorem}

\begin{proof}
We may suppose that $\cup C=V$. For $\{V\}$ we have
$\zeta^{-1}(\{V\})=(-1)^{|V|}+1$ by deletion-contraction, so
$\{V\}$ is eulerian if and only if $|V|$ is odd. Suppose the
statement is true for all clutters with $n>1$ edges and let
$C=\{e_0,e_1,\ldots,e_n\}$ be a clutter with $n+1$ edges. Since
the property of being a clique complex of a chordal graph is
hereditary on full subcomplexes it is sufficient to prove that
$\zeta^{-1}(C)=0$. But
\[\zeta^{-1}(C)=-\zeta^{-1}(C/e_0),\] by deletion-contraction since
$\zeta^{-1}(C-e_0)=0$ by induction. Let $\sigma$ be the simplex on
the set $\{e\in C-e_0 | e\cap e_0\neq\emptyset\}$. We have

\[N(C/e_0)=N(C-e_0)\cup\sigma.\] The proof follows from the fact
that $C/e_0$ satisfies conditions of theorem.
\end{proof}

The converse of Theorem \ref{clique} is not true in general. The
following example is due to N. Erokhovets.

\begin{example}
Let $C=\{\{1,2,3,6,7\},\{3,4,5,7,8\},\{1,2,3,4,5\}\}$ or, more
generally, $C=\{e_1,e_2,e_3\}$ such that $|e_i|$-odd for all $i$,
$e_3\varsubsetneq e_1\cup e_2$, $|e_1\cap e_2|$-even, $|e_1\cap
e_2\cap e_3|, |e_1\cap e_3|, |e_2\cap e_3|$-odd. $C$ is not odd
since $|e_1\cap e_2|$-even. For any proper set $I\varsubsetneq V$
we have $C|_I$ is either discrete, or it consists of one odd set,
or it consists of two odd sets with odd intersection. So $C$ is
eulerian if and only if $\zeta^{-1}(C)=0$. We have
\[
\zeta^{-1}(C)=\zeta^{-1}(C-e_3)-\zeta^{-1}(C/e_3)=0,\]since both
deletion and contraction is a clutter of two odd sets with even
intersection.
\end{example}

The following examples show that even conditions $(i)$ and $(ii)$
of Theorem \ref{clique} have not to be satisfied by eulerian
clutters.

\begin{example}
Let $C_1=\{e_1,e_2,e_3,e_4\}$, where $e_1=\{1,2,3,4,5\},
e_2=\{5,6,7,8,9\}, e_3=\{3,4,5,6,7\}$ and $e_4=\{1,9,10\}$ be a
clutter on $10$ vertices. The nerve $N(C_1)$ is not a flag complex
since $\{e_1,e_2,e_4\}$ is a minimal nonface. By simple inspection
we see that $C_1$ is eulerian iff $\zeta^{-1}(C_1)=0$, which is
satisfied, since
\[\zeta^{-1}(C_1)=\zeta^{-1}(C-e_3)-\zeta^{-1}(C/e_3)=0.\]
Similarly, let $C_2=\{e_1,e_2,e_3,e_5,e_6\}$, where
$e_5=\{1,10,11\}$ and $e_6=\{9,11,12\}$, be a clutter on $12$
vertices. Then it is easy to see that $C_2$ is eulerian, but
$G(C_2)$ is not a chordal graph.
\end{example}

Nevertheless the converse of Theorem \ref{clique} holds for an
eulerian clutters with an additional condition.

\begin{theorem}\label{converse}
Let $C$ be an eulerian clutter such that
\[(\ast) \ \ \ \ \   e\varsubsetneq\cup (C-e), \ \mbox{for all} \ e\in C.\]
Then $C$ is odd and the nerve $N(C)$ is the clique complex of the
chordal graph $G(C)$.
\end{theorem}
\begin{proof}
Note that $(\ast)$ is equivalent to the condition that any
subclutter $C'\varsubsetneq C$ is a proper restriction of $C$,
i.e. for the sets of vertices holds $V(C')\varsubsetneq V(C)$.

Let $C'\subset C$ be a minimal face of $N(C)$ such that $\cap C'$
is even. Then by induction and deletion-contraction follows
$\zeta^{-1}(C')=(-1)^{|C'|-1}2$. We have that either $C'=C$ or
$C'\varsubsetneq C$. In both cases $C$ is not eulerian.

\begin{itemize}

\item[Case 1:] Suppose $C'=\{e_1,\ldots,e_n\}$ is a minimal
nonface of $N(C)$ and $n>2$. Then $C'-e_1$ satisfies conditions of
Theorem \ref{clique}, so $\zeta^{-1}(C')=-\zeta^{-1}(C'/e_1)$ by
deletion-contraction. But $N(C'/e_1)$ is a minimal face of
$N(C/e_1)$ such that $\cap (C'/e_1)$ is even, which implies
$\zeta^{-1}(C'/e_1)=(-1)^{n-2}2$.

\item[Case 2:] Let $C'=\{e_1,\ldots,e_n\}\subset C$ be an odd
clutter such that $N(C')$ is the $n$-cycle with no chords. We find
by induction and deletion-contraction that
$\zeta^{-1}(C')=(-1)^{n-1}2$.
\end{itemize}

In both cases it follows from $(\ast)$ that either $C'=C$ or $C'$
is a proper restriction of $C$, which contradicts the condition
$C$ is eulerian.

\end{proof}

Let $C$ be a clutter on the vertex set $V$. For a subclutter
$S\subset C$, let $\lambda(S)\vdash|V|$ be the partition of the
number of vertices $|V|$ whose parts are equal to the numbers of
vertices of connected components of $S$. The chromatic symmetric
function $\Psi(C)$ of a clutter $C$ has the following expansion in
the power sum basis of symmetric functions (Theorem 3.5,
\cite{SS})

\[\Psi(C)=\sum_{S\subset C}(-1)^{|S|}p_{\lambda(S)}.\]

\begin{proposition}
The canonical morphism $\Psi:\mathcal{H}(\mathbf{C})\rightarrow
Sym$ is an epimorphism.
\end{proposition}
\begin{proof}
Let $D_{[n]}=\{[n]\}$ be a clutter on $[n]=\{1,\ldots,n\}$. For a
partition
$\lambda=\{\lambda_1\geq\lambda_2\geq\cdots\geq\lambda_k\}$ set
$D_{\lambda}=D_{[\lambda_1]}\sqcup
D_{[\lambda_2]}\sqcup\ldots\sqcup D_{[\lambda_k]}$. We have
\[p_\lambda=\prod_{i=1,k}p_{\lambda_i}=\prod_{i=1,j}(\Psi(D_{\lambda_i})-\Psi(D_{[\lambda_i]}))\cdot\Psi(D_{k-j}),\]
where $j=\max\{i | \lambda_i>1\}$. On the other hand
\[\Psi(D_\lambda)=\prod_{i=1,k}\Psi(D_{[\lambda_i]})=\prod_{i=1,j}(p_{1^{\lambda_i}}-p_{\lambda_i})\cdot
p_{1^{k-j}},\] which shows that the transition matrix of bases
$\{p_\lambda\}$ and $\{\Psi(D_\lambda)\}$ is an involution.
\end{proof}

If a clutter $C$ satisfies the generalized Dehn-Sommerville
relations, i.e. $C\in
S_-(\mathcal{H}(\mathbf{C}),\zeta_\mathbf{C})$ then $\Psi(C)\in
S_-(Sym,\zeta_S)$. We say that a partition $\lambda$ is odd if all
its parts are odd. The odd subalgebra of symmetric functions is
linearly spanned by the set of all odd power sum symmetric
functions, (Proposition 7.1, \cite{ABS})

\[S_-(Sym,\zeta_S)={\rm span}\{p_\lambda|\lambda \ \mbox{is
odd}\}.\]

\begin{theorem}\label{coincide}
The clutter $C$ is eulerian if the partition $\lambda(S)$ is odd
for any subclutter $S\subset C$.
\end{theorem}

\begin{proof}

Particulary, we have that $\lambda(C')$ is an odd partition for
any simplex $C'\in N(C)$. Consequently the clutter $C$ is odd. We
show that $C$ satisfies conditions of Theorem $\ref{clique}$:
\begin{itemize}
\item[(i)] Suppose there is a nonchordal cycle $e_1,\ldots,e_k,
k>3$ in the intersection graph $G(C)$. As
$\lambda(\{e_1,\ldots,e_k\})$ is an odd partition, we conclude
that $|e_1\cup\ldots\cup e_k|$ is odd. On the other hand

\[|e_1\cup\ldots\cup e_k|=|e_1|+\cdots+|e_k|-|e_1\cap
e_2|-|e_2\cap e_3|-\cdots-|e_k\cap e_1|, \] which is even.

\item[(ii)] Suppose that $\{e_1,\ldots,e_k\}, k>2$ is a minimal
nonface of the nerve $N(C)$. Then

\[0=|e_1\cap\ldots\cap e_k|=|e_1\cup\ldots\cup
e_k|-\sum_{j=1}^{k-1}(-1)^{j-1}\sum_{i_1<\cdots<i_k}|e_{i_1}\cap\ldots\cap
e_{i_k}|.\]

But any of $2^k-1$ summands on the righthand side is odd since the
clutter $C$ is odd.
\end{itemize}
Hence the nerve $N(C)$ is a flag complex, i.e. it is the clique
complex of the graph $G(C)$.

\end{proof}

\begin{remark}
The eulerian clutters generate the proper subalgebra of the odd
subalgebra of clutters. There are clutters which satisfy the
generalized Dehn-Sommerville relations, but which are not
eulerian. The simple example is given by the clutter on five
vertices $C=\{\{1,2,3\}, \{2,3,4\}, \\ \{1,3,4,5\}\}$. The clutter
$C$ and the eulerian clutter $C'=\{\{1,2,3\},\{3,4,5\}\}$ have the
same chromatic symmetric function

\[\Psi(C)=\Psi(C')=p_{1,1,1,1,1}-2p_{3,1,1}+p_5.\]

\end{remark}

\section{Hopf algebra of simplicial complexes}

An abstract simplicial complex $K$ on the vertex set $V=V(K)$ is a
lower ideal in the boolean poset $\mathcal{P}(V)$ of all subsets
of $V$. An element $\sigma\in K$ is called a face of $K$. We
require that $\{v\}\in K$ for all vertices $v\in V$. The induced
subcomplex $K|_I$ of $K$ on the vertex set $I\subset V(K)$ is a
family $K|_I=\{\sigma\subset I | \sigma\in K\}$. The join of
simplicial complexes $K_1$ and $K_2$ is the complex $K_1\ast
K_2=\{\sigma\cup\tau\subset V(K_1)\sqcup V(K_2) | \sigma\in K_1,
\tau\in K_2\}$. A simplicial isomorphism is a bijection
$f:V(K_1)\rightarrow V(K_2)$ such that $\sigma\in K_1$ if and only
if $f(\sigma)\in K_2$.

The {\it independence complex} $\mathrm{Ind}(C)$ of a clutter $C$
is a simplicial complex $\mathrm{Ind}(C)=\{I\subset V(C) | C|_I \
\mbox{is discrete}\}.$ A simplicial complex $K$ determines the
clutter $C(K)$ of minimal nonfaces of $K$. These assignments are
mutually inverse $\mathrm{Ind}(C(K))=K, C(\mathrm{Ind}(C))=C$ and
satisfy the following properties

\begin{equation}\label{ind}
\mathrm{Ind}(C_1\sqcup C_2)=\mathrm{Ind}(C_1)\ast\mathrm{Ind}(C_2)
\ \ \mbox{and} \ \ \mathrm{Ind}(C)|_I=\mathrm{Ind}(C|_I),
\end{equation}
for all clutters $C_1$ and $C_2$ and for any $I\subset V(C)$.

Let $\mathcal{H}(\mathbf{K})=\oplus_{n\geq
0}\mathcal{H}_n(\mathbf{K})$ be the graded $\mathbf{k}-$vector
space linearly spanned by the set $\mathbf{K}$ of all isomorphism
classes of finite simplicial complexes. The grading is given by
the number of vertices. Let
$\mathrm{Ind}:\mathcal{H}(\mathbf{C})\rightarrow\mathcal{H}(\mathbf{K})$
be the linear extension of the map that associates the
independence complex $\mathrm{Ind}(C)$ to a clutter $C$. By the
properties $(\ref{ind})$ the space $\mathcal{H}(\mathbf{K})$ is
equipped with the Hopf algebra structure with respect to which the
map $\mathrm{Ind}$ is a Hopf algebra isomorphism. The product on
$\mathcal{H}(\mathbf{K})$ is defined by the join $K_1\ast K_2$ and
the coproduct by induced subcomplexes $\Delta(K)=\sum_{I\subset
V(K)}K|_I\otimes K|_{V(K)\setminus I}$. The space
$\mathcal{H}(\mathbf{K})$ is a graded connected commutative and
cocommutative Hopf algebra. The unit element is the simplicial
complex $K_\emptyset$ on the empty set of vertices and the counit
is $\epsilon(K)=\left\{\begin{array}{ll} 1, & K=K_\emptyset \\ 0,
& \mbox{otherwise}\end{array}\right..$ The antipode is determined
by $(\ref{antiph})$

\[S(K)=\sum_{k\geq 1}(-1)^k\sum_{I_1\sqcup\ldots\sqcup
I_k=V(K)}K|_{I_1}\ast\cdots\ast K|_{I_k},\]where the sum goes over
all ordered decompositions of the vertex set.

The map $\mathrm{Ind}$ become an isomorphism of CHA's if we take
$\zeta_\mathbf{K}(K)=\left\{\begin{array}{ll} 1, & K \ \mbox{is a simplex} \\
0, & \mbox{otherwise}\end{array}\right.$.

A {\it partition} of the simplicial complex $K$ is the ordered
decomposition $(I_1,\ldots,I_k)$ of the vertex set $V(K)$ such
that all induced subcomplexes $K|_{I_1},\ldots,K|_{I_k}$ are
simpleces. For a composition $\alpha=(a_1,\ldots,a_k)\models n$
the coefficient $(\zeta_\mathbf{K})_\alpha(K)$ is the number of
partitions of $K$ whose parts have the prescribed numbers of
elements $|I_j|=a_j, 1\leq j\leq k$. The canonical morphism
$\Psi:(\mathcal{H}(\mathbf{K}),\zeta_\mathbf{K})\longrightarrow(QSym,\zeta_Q),$
given by $(\ref{canonical})$ with

\[\Psi(K)=\sum_{\alpha\models
n}(\zeta_\mathbf{K})_\alpha(K)M_\alpha, \
K\in\mathcal{H}_n(\mathbf{K}),\] assigns the generating function
of all partition functions to a simplicial complex $K$. A {\it
partition function} of a simplicial complex $K$ is the map
$f:V(K)\rightarrow\mathbb{N}$ such that $f^{-1}(\{i\})$ is a face
of $K$ for all $i\in f(V(K))$. Then $\Psi(K)=\sum_{f}\prod_{v\in
V(K)}x_{f(v)},$ where the sum is over all partition functions.

The principal specialization
$\chi(K,m)=\zeta_\mathbf{K}^m(K)=\Psi(K)(1^m)$ defines the {\it
partition polynomial} of a simplicial complex $K$ which counts the
number of partition functions $f:V(K)\rightarrow[m]$. The value
$\chi(K,-1)$ is of special interest and by the antipode formula we
have

\[\chi(K,-1)=\zeta_\mathbf{K}^{-1}(K)=\zeta_\mathbf{K}\circ S(K)=\sum_{\alpha\models
n}(-1)^{k(\alpha)}(\zeta_\mathbf{K})_\alpha(K).\]

The value of the Euler character
$\chi_{\mathbf{K}}=\overline{\zeta}_\mathbf{K}\zeta_\mathbf{K}$ at
a simplicial complex $K$ is given by
\[\chi_{\mathbf{K}}(K)=\sum_{(I,I^c) \ \mbox{partition}}(-1)^{|I|}.\]

\begin{definition}
A simplicial complex $K$ is eulerian if
$\chi_\mathbf{K}(K|_I)=\epsilon(K|_I)$ for all $I\subset V(K)$.
\end{definition}

As in the case of hypergraphs we have an equivalent
characterization

\begin{lemma}
A simplicial complex $K\neq K_\emptyset$ is eulerian if and only
if either $K|_I$ is a simplex or $\zeta_\mathbf{K}^{-1}(K|_I)=0,$
for all $\emptyset\neq I\subset V(K)$.
\end{lemma}

The eulerian subalgebra $\mathcal{E}(\mathbf{K})$ is generated by
all isomorphism classes of eulerian simplicial complexes. It is a
subalgebra $\mathcal{E}(\mathbf{K})\subset
S_-(\mathcal{H}(\mathbf{K}), \zeta_\mathbf{K})$ of the odd algebra
of CHA of simplicial complexes. Consequently the relations
$(\ref{relation})$ are satisfied by coefficients
$(\zeta_\mathbf{K})_\alpha, \alpha\models |V(K)|$  of the eulerian
simplicial complexes $K\in\mathcal{E}(\mathbf{K})$.

\begin{example}
Suppose that $\{i,j\}$ is a minimal nonface of a simplicial
complex $K$. Then $\chi(K|_{\{i,j\}})=-2,$ so $K$ is not eulerian.
The nessesary condition for $K$ to be an eulerian is that its
1-skeleton is the complete graph on $V(K)$.
\end{example}

The isomorphism $\mathrm{Ind}$ gives the complete characterization
of eulerian simplicial complexes.

\begin{theorem}
A simplicial complex $K$ is eulerian $K\in\mathcal{E}(\mathbf{K})$
if and only if it is the independence complex $K=Ind(C)$ of an
eulerian clutter $C\in\mathcal{E}(\mathbf{C})$ or equivalently if
and only if its clutter $C(K)$ of minimal nonfaces is eulerian
$C(K)\in\mathcal{E}(\mathbf{C})$.
\end{theorem}

\begin{example}
1. Let $\Delta[n]$ be a simplex on $n$ vertices. The boundary
complex $\partial\Delta[n]$ is eulerian if and only if $n$ is odd.

2. Let $C$ be a clutter such that the nerve complex $N(C)$ is a
tree. Then the independence complex $\mathrm{Ind}(C)$ is eulerian
if and only if $C$ is odd.

\end{example}

\section*{Acknowledgments}

The first two authors were supported by Ministry of Science of
Republic of Serbia, project 174034. We thank to N. Erokhovets for
valuable comments and for appointing us to the mistake in the
formulation of \cite[Theoerem 8.11]{GS}.




\bibliographystyle{model1a-num-names}
\bibliography{<your-bib-database>}

\begin{thebibliography}{9}

\bibitem{AB} A. A. Ayzenberg, V. M. Buchstaber, Nerve Complexes and Moment-Angle Spaces of Convex
Polytopes, Tr. Mat. Inst. im. V.A. Steklova, Ross. Akad. Nauk 275
(2011) 22-–54. [Proc. Steklov Inst. Math. 275 (2011) 15--46.]

\bibitem{ABS} M. Aguiar, N. Bergeron, F. Sottile, Combinatorial Hopf
algebras and generalized Dehn-Sommerville relations, Compositio
Mathematica 142 (2006) 1-30.

\bibitem{BB} M. Bayer, L. Billera, Generalized Dehn-Sommerville relations for polytopes, spheres,
and Eulerian partially ordered sets, Invent. Math. 79 (1985)
143-157.

\bibitem{GS} V. Gruji\'c, T. Stojadinovi\'c, Hopf
algebra of building sets, Electron. J. Combin. 19(4) (2012) P42.

\bibitem{JR} S. Joni, G.-C. Rota, Coalgebras and bialgebras in combinatorics, Stud. Appl.
Math. 61 (1979) 93-139.

\bibitem{EC} R. Stanley, Enumerative combinatorics. Vol. 2, Cambridge Univ. Press, Cambridge,
1999.

\bibitem{SS} R. Stanley, Graphs colorings and related symmetric
functions: Ideas and Apllications, Discrete Math. 193 (1998)
267-286.

\end{thebibliography}







\end{document}